\documentclass[10pt]{article}
\usepackage{amsmath, amsthm, amsfonts, amssymb, amscd, bm, dsfont, enumitem}
\usepackage{graphicx}
\usepackage{subfigure}
\usepackage{algorithm}
\usepackage{algorithmic}

\usepackage[pdftex,
            bookmarks=true,
            bookmarksnumbered=true,
            pdfpagemode=UseOutlines,
            pdfstartview=FitH,
            pdfpagelayout=OneColumn,
            colorlinks=false,
            pdfborder={0 0 0}]{hyperref}
\usepackage{comment}

\theoremstyle{plain}
\newtheorem{theorem}{Theorem}
\newtheorem{lem}[theorem]{Lemma}
\newtheorem{thm}[theorem]{Theorem}

\theoremstyle{definition}
\newtheorem{defn}{Definition}

\newtheorem{exmp}[theorem]{Example}

\numberwithin{theorem}{section}
\numberwithin{defn}{section}
\numberwithin{equation}{section}

\newcommand{\real}{\mathbb{R}}

\newcommand{\myobot}{\mathop{\bigcirc\kern-11.5pt\perp}\nolimits}

\makeatletter
\def\@tempa#1{\@xp\@tempb\meaning#1\@nil#1}
\def\@tempb#1>#2#3 #4\@nil#5{%
  \@xp\ifx\csname#3\endcsname\mathaccent
    \@tempc#4?"7777\@nil#5%
  \else
    \PackageWarningNoLine{amsmath}{%
      Unable to redefine math accent \string#5}%
  \fi
}
\def\@tempc#1"#2#3#4#5#6\@nil#7{%
  \chardef\@tempd="#3\relax\set@mathaccent\@tempd{#7}{#2}{#4#5}}

\@tempa\widehat
\makeatother

\title{Error Estimation of Numerical Solvers for Linear Ordinary Differential Equations}

\author{Wenyuan Wu \and Wenqiang Yang}

\date{\today}

\begin{document}

\maketitle

\begin{abstract}
Solving Linear Ordinary Differential Equations (ODEs) plays an important role in many applications. There are various
numerical methods and solvers to obtain approximate solutions. However, few work about global error estimation can be found in
the literature. In this paper, we first give a definition of the residual, based on the piecewise Hermit interpolation, which is a kind of the backward-error of ODE solvers. It indicates the reliability and quality of numerical solution.
Secondly, the global error between the exact solution and an approximate solution is the forward error and a bound of it can be given by using the backward-error. The examples in the paper show that our estimate works well for a large class of ODE models.
\end{abstract}

%%%%%%%%%%%%%%%%%%%%%%%%%%%%%%%%%%%%%%%%%%%%%%%%%%%%%%%%%%%%%%%%
\section{Introduction}\label{sec:intro}
Numerical modelling and simulation is of great importance in mechanical engineering and industrial design. ODEs Often appear in such models. To simplify the model or ease solving cost, people usually consider linearized systems.
Numerical solving of linear ODE is well studied. Even a number of efficient and sophisticated solvers have been developed. How to qualify and how to measure the reliability of these solvers is a natural question.

There are numerous studies about estimating and controlling errors by constant step size methods and variation of step size methods. Please see the survey \cite{Shampine2004} for more details.
Different from the interests above, the motivation of this paper is to establish a bridge to connect the forward error (the global error) and the backward error (the residual) for a given numerical solution of ODEs by some method.

The standard from of an initial value problem (IVP) for ODEs is
\begin{equation}\label{sivp}
\left\{\begin{array}{l}
{\bm x}'={\bm f}(t,{\bm x}(t))\\
{\bm x}(t_{0})={\bm x}_{0}
\end{array}\right.
\end{equation}
Where ${\bm x}={\bm x}(t):\real\rightarrow\real^n$
is vector of solutions as a function time. At the starting
point $t_{0}$, ${\bm x}(t_{0})={\bm x}_{0}$ is the given as
the initial condition, in which ${\bm x}_{0}\in\real^n$,
and ${\bm f}:\real\times\real^n\rightarrow\real$ is in
 general a nonlinear function of $t$ and ${\bm x}$.

The existence, uniqueness and continuity for exact solution are three general requirement of well-posed problems, which for initial value problem were rigorously treated by Cauchy and Walter \cite{Walter1998}. Traditional discrete numerical methods partition the interval of interest, $[t_{0},t_{m}]$, by introducing a mesh $t_{0}< t_{1}<\cdots<t_{N}=t_{m}$ and generate a discrete approximation ${\bm x}_{i}\approx {\bm x}(t_{i})$ for each associated mesh-point. There must be iterative approximation error while adopting numerical methods \cite{Enright2000}. Moreover, the more mesh-points are selected, the smaller error and more cost will be. The most important reason for estimating error is to gain some confidence in the numerical solution.

Residual control has a number of virtues. One is that it assesses the error throughout the interval of integration, not just the error made in advancing a step \cite{Shampine2004}.
The schemes proposed by Enright \cite{Enright2000} for estimating a norm of the residual are plausible but without considering the derivative of each point and they are not asymptotically correct. Others \cite{Higham1989, Higham1991, Shampine2005} do much works on obtaining an asymptotically correct estimate, most of them work hard on the local errors.

Usually, we can not obtain the exact solution of IVP for ODEs. Fortunately, a number of numerical solvers e.g. MATLAB can approach it when the used tolerance is sufficiently small. And many people have done important work on this \cite{Calvo1997, Shampine1989, Stetter1980}. However, no matter how many mesh-points you choose, for numerical solution, there always has a question about the global error which tells how close from the numerical solution to the exact solution at any point. Global error estimation \cite{Dormand1984, Dormand1991, Shampine1976, Shampine19762} can answer it partially, but it is generally thought to be too expensive and too restricted. Lipschitz constant \cite{Robert2013} can help create an error bound, but the bound is too pessimistic in many applications.  In \cite{Robert2013}, it is analyzed the relationship of condition number and residual to estimate global error, but the condition number approach needs the fundamental solutions of homogenous ODEs.

In this paper, our main contribution is a global error estimation method for a given numerical solution of a linear ODE, which works even for unstable ODEs.

\section{Preliminaries}\label{sec:res}
In this section, to deal with numerical solution of dynamical system, we firstly give some brief introduction on
initial value problem of linear ODEs and an interpolation method -- hermit cubic spline \cite{Uri2009, Endre2003}. Then we give a
definition of residual-error similar to Sec.\ref{sec:res} of \cite{Robert2013}.

\subsection{Initial Value Problem of Linear ODEs}\label{ssec:ivp}
Usually, we first translate nonlinear systems
to linear or linearized systems, as ${\bm f}$ is a velocity
vector and ${\bm x}$ is a curve in phase space that tangent
to the vector field at every point. The system (\ref{sivp})
can also be expressed into matrix-vector notation conveniently as follows:
\begin{equation}\label{livp}
\left\{\begin{array}{l}
 {\bm x}'={\bm A}(t)\cdot{\bm x}(t)+{\bm q}(t)\\
 {\bm x}(t_{0})={\bm x}_{0}
\end{array}\right.
\end{equation}

Where the $n\times n$ matrix ${\bm A}(t)$ is the state matrix,
the $n$ dimensional vector ${\bm q}(t)$ corresponds to the inhomogeneous
part of the system.

In addition, for a linear time-invariant system,
its special solution is given by Equation (\ref{eq:sol}) (see  $18.VI$ of \cite{Walter1998}).
\begin{equation}\label{eq:sol}
\left\{\begin{array}{l}
{\bm x}(t)={\bm X}(t)\cdot[{\bm c}+\int^{t}_{t_{0}}{\bm X}^{-1}(s)\cdot{\bm q}(s)ds]\\
{\bm x}(t_{0})={\bm c}
\end{array} \right.
\end{equation}
where $n\times n$ matrix ${\bm X}(t)$ is the fundamental solution
of ${\bm x}'={\bm A}(t)\cdot{\bm x}(t)$, ${\bm c}$ is a vector of constant.
But there are many linear time-variant systems in practice, which is difficult to calculate the fundamental solutions ${\bm X}(t)$ directly.

\subsection{Hermite Cubic Spline}\label{ssec:phi}
 A number of numerical solvers for especially IVP for ODEs are
 given in MATLAB, such as: ode$45$, ode$23$, ode$113$, ode$23t$, ode $25s$
 and so on. Ode$45$ is the most common used method for solving ODEs,
 adopting Runge-Kutta fourth-order algorithm, suitable for high precision.

However, there is a question that how close is the numerical solution from
the exact solution. We can only set up the precision of each step by MATLAB, and we
can also estimate the local truncation error which can roughly reflect
the distance.
 In fact, we can get not only the values of the variables of Eq.(\ref{livp})
 at each step, but also their derivatives by ode$45$ directly. In order
 to construct continuous functions to approximate the exact solution, we apply
 Hermite cubic spline interpolation.

We strengthen our requirements on the smoothness of the functions that we wish to interpolate and assume that
$f\in C^{1}[a,b]$; simultaneously, we shall relax the smoothness requirements
on the result spline approximation $p$ by demanding that $p\in C^{1}[a,b]$
only \cite{Endre2003}.

 \begin{thm}\label{thm:cubicHermit}
 Let $K={x_{0},...,x_{m}}$ be a set of knots in the interval $[a,b]$ with
 $a = x_{0}< x_{1} <...< x_{m} = b$ and $m\geq2$, there is a unique cubic spline $p \in C^{1}[a,b]$ on
 the interval $[x_{i-1},x_{i}]$, for $i = 1,...,m$ such that
 \[
    p(x_{i})= f(x_{i}), p'(x_{i}) = f'(x_{i}), for i = 1,...,m
 \]
 Writing the spline $p$ on the interval $[x_{i-1},x_{i}]$ as
 \[
    p(x)=c_{0}+c_{1}(x-x_{i-1})+c_{2}(x-x_{i-1})^{2}+c_{3}(x-x_{i-1})^{3}, x\in[x_{i-1},x_{i}]
 \]
 Where $c_{0}=f(x_{i-1})$, $c_{1}=f'(x_{i-1})$, \\
 and $c_{2}=3\frac{f(x_{i})-f(x_{i-1})}{(x_{i}-x_{i-1})^{2}}-\frac{f'(x_{i})+2f'(x_{i-1})}{(x_{i}-x_{i-1})}$,
 $c_{3}=\frac{f'(x_{i})+f'(x_{i-1})}{(x_{i}-x_{i-1})^{2}}-2\frac{f(x_{i})-f(x_{i-1})}{(x_{i}-x_{i-1})^{3}}$.
 \end{thm}

 \begin{thm}\label{thm:cubicHermiterror}
 Let $f \in  C^{4}[a,b]$, and let s be the Hermite cubic spline that interpolates $f$
 at the knots $a = x_{0}< x_{1} <...< x_{m} = b$; then, the following error bound holds:
 \[
 \|f-p\|_{\infty}\leq\frac{1}{384}h^{4}\|f^{(4)}\|_{\infty}
 \]
 where $f^{(4)}$ is the fourth derivative of $f$ with respect to its
 argument $x$, $h=\max\limits_{i}(x_{i}-x_{i-1})$, and $\|\cdot\|_{\infty}$ denotes
 the $\infty$-norm on the interval $[a,b]$.
  \end{thm}

\subsection{The Residual of a Numerical Solution}\label{ssec:re}
When a numerical solver deals with IVP for ODEs (\ref{sivp}) for $t \in[t_{0},t_{m}]$,
it divides the interval $[t_{0},t_{m}]$ into $m$ subintervals which depends by required accuracy.
Associate with Eq.(\ref{sivp}), we give a definition to backward error.

\begin{defn}\label{defn:backward error}
Let $\{t_{0},t_{1},...,t_{m}\}$ be a set of knots in the interval $[t_{0},t_{m}]$
with $t_{0}< t_{1} <...< t_{m}$, and $m\geq2$, the length of each step $h_{i}$ can be define as
 $h_{i}=t_{i}-t_{i-1}$, $m\geq i\geq1$.
 Let  $\{{\bm x}_{k}\}$ and $\{{\bm x}'_{k}\}$ for $k=0,...,m$ be the numerical solution of IVP for odes Eq.(\ref{sivp}), where ${\bm x}_{k}$, ${\bm x}'_{k}$ are vectors.
 There is a unique Hermite cubic spline
$\widetilde{{\bm x}}(t)\in C^1[t_{0},t_{m}]$, at each of the knot $k$ for $k=0,...,m$, such that
\[
\widetilde{{\bm x}}(t_{k})={\bm x}_{k},\widetilde{{\bm x}}'(t_{k})={\bm x}'_{k}
\]
\end{defn}
We define $\widetilde{{\bm x}}(t)$ as the corresponding \textbf{interpolating solution}.
Specially,  because of IVP for ODEs, the initial value $t_{0}$ is no error
point with no doubt, that means $\widetilde{{\bm x}}(t_{0})={\bm x}_{0}$, that's to say, there is no error for $\widetilde{{\bm x}}$ at the initial point.

\begin{defn}\label{defn:residual error}
Let $n$-vector $\widetilde{{\bm x}}(t) \in C^1[t_{0},t_{m}]$ be interpolating solution of Eq.(\ref{sivp}).
Then its \textbf{residual} is defined to be
\begin{equation}\label{re:resdual}
{\bm\delta}(t)=\widetilde{{\bm x}}'(t)-f(t,\widetilde{{\bm x}}(t))
\end{equation}
Moreover, it's clear that ${\bm\delta}(t_{0})=0$.
\end{defn}

\begin{exmp} \label{ex:1}
Consider this initial-value problem:
\[
   x(t)'=x(t), x(0)=1
\]
Obviously,the exact solution of this system is $x^{*}(t)=e^{t}$.
Suppose $t\in[0, 2]$ and we simply divide the interval
into $2$ parts $[0, 1]$ and $[1, 2]$ with the same step length $h=1$.
Then we have the numerical solutions ${x_{0}=1, x_{1}=e, x_{2}=e^{2}}$
and ${x'_{0}=1, x'_{1}=e, x'_{2}=e^{2}}$ respectively.
Then we can interpolate Hermite cubic spline, as follow:
\[
   \widetilde{x}(t)= \left\{\begin{array}{ccc}
   &1+t+(2e-5)t^{2}+(3-e)t^{3}& t\in[0,1] \\
   &e\cdot(t+(2e-5)(t-1)^{2}+(3-e)(t-1)^{3})& t\in[1,2]
   \end{array}\right.
\]
In addition, the forward error can be write as $\Delta x(t)=x^{*}(t)-\widetilde{x}(t)$.
\begin{figure}[H]

  \includegraphics[height=8cm,width=13cm]{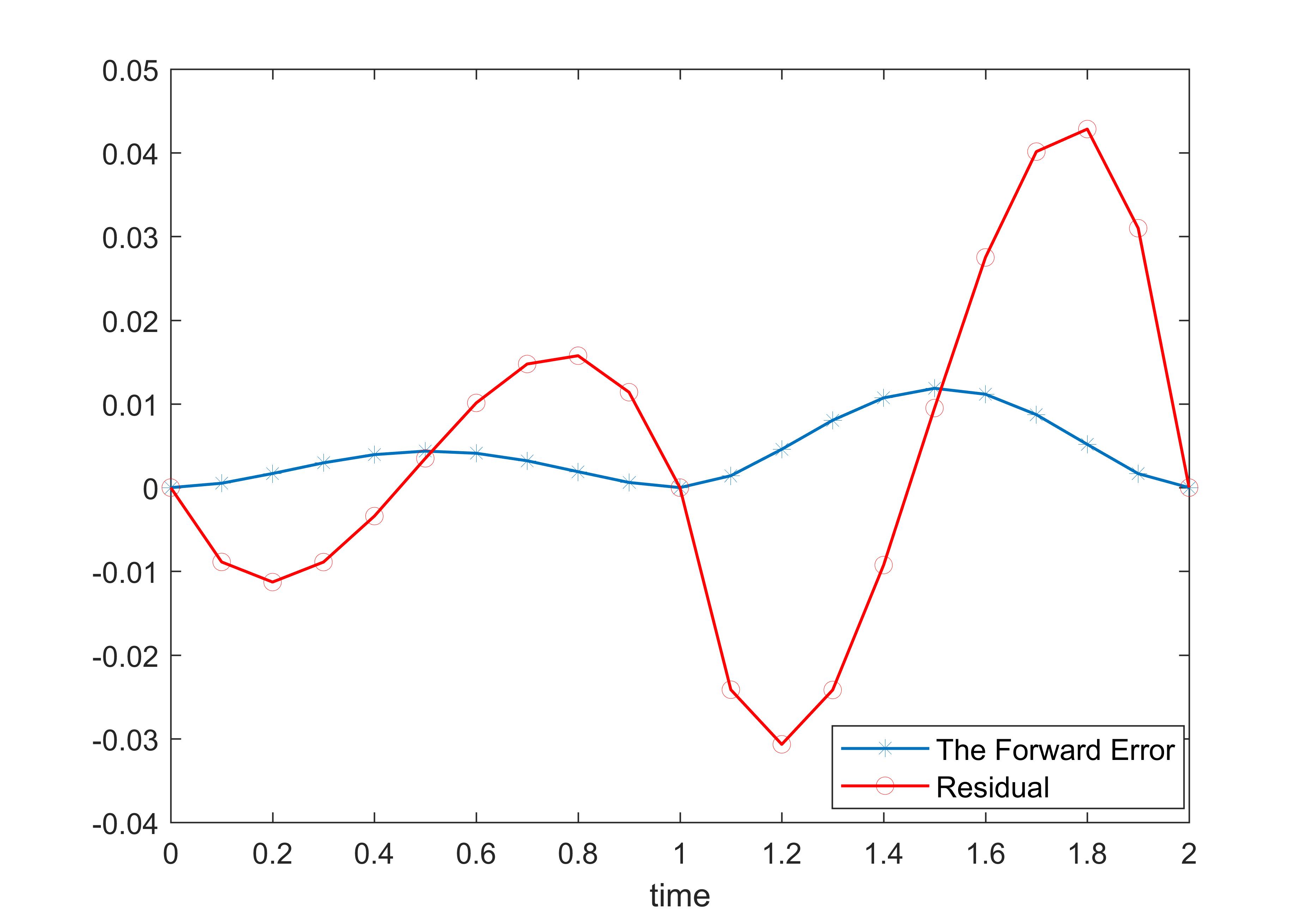}\\
  \caption{The Performance of Error Expressions}\label{example}
\end{figure}

According to Eq.(\ref{re:resdual}), it's easy to get the residual
\[
   \delta(t)=\left\{\begin{array}{ccc}
   &((-3+e)\cdot t -4e+11)\cdot(t-1)\cdot t& t\in[0,1] \\
   &e\cdot((-3+e)\cdot t -5e+14)\cdot(t-1)\cdot (t-2) & t\in[1,2]
    \end{array}\right.
\]

Above all, we can draw the figure \ref{example} of the performance of error expressions
by Matlab. Obviously, ${\delta}(0)=0$, ${\delta}(1)=0$, ${\delta}(2)=0$, that means Hermite cubic spline accurately across three points, besides these points, the residual will be greater than zero.
\end{exmp}

%%%%%%%%%%%%%%%%%%%%%%%%%%%%%%%%%%%%%%%%%%%%%%%%%%%%%%%%%%%%%%%%
\section{Error Estimation for Linear ODEs}\label{sec:error estimation}

In this section,  we firstly give some basic concepts of the linear ODEs in Sec.\ref{ssec:error relationship}, in addition, in Sec.\ref{ssec:ee}, we give an error estimation for linear ODEs system.

\subsection{Global Error}\label{ssec:error relationship}

Let $n$-vector ${\bm x}^{*}(t)$ be the exact solutions of linear ordinary system Eq.(\ref{livp}), such that
\begin{equation}\label{exact solution of lvip}
	({\bm x}^{*})'={\bm f}(t,{\bm x}^{*})={\bm A}(t)\cdot{\bm x}^{*}+{\bm q}(t)
\end{equation}
Let $\widetilde{{\bm x}}(t)$ be an interpolating vector solution from a given numerical solution data.

\begin{defn}
Let $n$-vector $\Delta{\bm x}(t) = {\bm x}^{*} - \widetilde{{\bm x}}$ be the \textbf{forward error}.
\end{defn}

It is not difficult to see the following result.
\begin{lem}\label{lem:forward}
Let ${\bm \delta}(t)$ be the residual of  Eq.(\ref{livp}), such that $ {\bm\delta}=\widetilde{{\bm x}}'-{\bm A}(t)\cdot \widetilde{{\bm x}}-{\bm q}(t)$. Then the forward error satisfies the following IVP of ODE
\begin{equation}\label{eq:forward error solution}
	\Delta{\bm x}'={\bm A}(t)\cdot \Delta{\bm x}-{\bm \delta}
\end{equation}
and $\Delta{\bm x}(t_0)=0$.
\end{lem}

\subsection{Error Estimation}\label{ssec:ee}
Since the system is a linear time-variant system, the state matrix $A(t)$ makes the differential equation difficult of solve exactly. We need to suppose the change of $A(t)$ is not so large.

\begin{defn}\label{def:taylor expand}[Taylor expansion]
Let $n\times n$ matrix ${\bm A}(t) \in C^{3} $ at the neighborhood of $t_{0}$, and let ${\bm A}_{0}={\bm A}(t_{0})$, ${\bm A}_{1}=\frac{d{\bm A}}{dt}|_{t=t_{0}}$, ${\bm A}_{2}=\frac{1}{2!} \cdot \frac{d^{2}{\bm A}}{dt^{2}}|_{t=t_{0}}$, respectively. Let ${\bm R}(t)$ be the high order remainder, such that
\begin{equation}\label{eq:taylor}
	{\bm A}(t)={\bm A}_{0}+{\bm A}_{1}\cdot(t-t_{0})+{\bm A}_{2}(t-t_{0})^2+{\bm R}(t)
\end{equation}

\end{defn}
It's clear that ${\bm A}_{0}$ is a $n\times n$ constant matrix. In dynamic system, the eigenvalue of ${\bm A}_{0}$ is usually real constant, in this paper, what we mainly deal with is this kind of problem.

For time-invariant system, which has exact solution, no matter what numerical solver we choose, we can do accurate error compensation. However, for time-variant system it is difficult to calculate the exact solution, so we need to give a boundary of forward error.

\begin{thm}\label{thm:main}
	Let $n\times n$ constant matrix ${\bm A}_{0}$ be linear part of Taylor expansion of ${\bm A}(t)$ at time $t=t_{0}$. Suppose  ${\bm A}_{0}$'s eigenvalues $\{\lambda_{i}$, $i=1,2,...,n\}$ are real, satisfying $\lambda_{i} \geq \lambda_{j}, i \leq j$. There is an invertible matrix ${\bm P}$, such that ${\bm A}_{0}=P\Sigma P^{-1}$, where $\Sigma=\left( \begin {array}{cccc}
                    \lambda_{1}&&&\\
                      & \lambda_{2}&&\\
                     &&\ddots&\\
                     &&&\lambda_{n}
                   \end {array} \right)$ is diagonal matrix.
Let $h_{max}=\max\limits_{i,j,t}{\{({\bm P}^{-1}({\bm A}(t)-{\bm A}_{0}){\bm P}})_{i,j}\}$, and $\delta_{max}=\max\limits_{i,t}{\{(-{\bm P}^{-1}{\bm \delta}(t))_{i}\}}$, where $t\in [t_{0},t_{m}]$.

Then we have
\begin{equation} \label{eq:result1}
	 \|\Delta{\bm x}\|_{\infty} \leq \|{\bm P}\|_{\infty}\cdot \frac{\delta_{max}}{\lambda_{1}+n\cdot h_{max}}(e^{(\lambda_{1}+n\cdot h_{max})t}-1)
	 	\end{equation}		
\end{thm}

\begin{proof}
As to Eq.({\ref{eq:forward error solution}}), after Taylor expansion, we can rewrite the equation as follow
\[
\begin{array}{rcl}
	\Delta{\bm x}'&=&{\bm A}_{0}\cdot \Delta{\bm x}+({\bm A}(t)-{\bm A}_{0})\cdot \Delta{\bm x}-{\bm \delta}\\
 &=& P\Sigma P^{-1}\cdot \Delta{\bm x}+({\bm A}(t)-{\bm A}_{0})\cdot \Delta{\bm x}-{\bm \delta}
\end{array}
\]
Multiplying the former equation on both sides simultaneously by matrix ${\bm P}^{-1}$. Let a vector ${\bm y}=(y_{1}(t),y_{2}(t),...,y_{n}(t))^{T}=P^{-1} \Delta{\bm x}$, there must be ${\bm y}'={\bm P}^{-1} \Delta{\bm x}'$ and ${\bm y}(t_{0})=0$, to simplify equation as
\[
{\bm y}'= \Sigma {\bm y}+[{\bm P}^{-1}({\bm A}(t)-{\bm A}_{0}){\bm P}]\cdot {\bm y}+(-{\bm P}^{-1}{\bm \delta}(t))
\]
Let $z(t)=\max\limits_{1\leq i\leq n}|y_{i}(t)|$, that's to say $z(t)=\|{\bm y}\|_{\infty}$. For every integer $i$, $1\leq i\leq n$, no matter $y_{i}(t)$ is positive or negative, there must be
\[
   |y'_{i}(t)|\leq \lambda_{1}z(t)+n\cdot h_{max}z(t)+\delta_{max}
\]
Moreover, because $z(t)$ is always one element of vector $|{\bm y}|$, there must be
\[
 z'(t)\leq \lambda_{1}z(t)+n\cdot h_{max}z(t)+\delta_{max}
 \]
Denote a one-dimensional function $\phi(t)$, where $\phi(t_{0})=0$, satisfying
\[
     \phi'(t)=(\lambda_{1}+n\cdot h_{max})\phi(t)+\delta_{max}
\]
It's easy to calculate the solution
 \[\phi(t)=\frac{\delta_{max}}{\lambda_{1}+n\cdot h_{max}}(e^{(\lambda_{1}+n\cdot h_{max})t}-1)\]

Additionally, we know $z(t)\leq \phi(t)$, and $\|\Delta{\bm x}\|_{\infty}\leq \|{\bm P}\|_{\infty}\cdot \|{\bm y}\|_{\infty}$, we can finally deduce Eq(\ref{eq:result1})
\end{proof}

Above all, the error estimation will exponentially increase by time, mainly depending on the the exponent $\lambda_{1}+n\cdot h_{max}$, and linearly related to $\delta_{max}$. It seems that if dynamic system is stable $\lambda_{1}+n\cdot h_{max}<=0$, the error estimation boundary will be a constant while time is long enough. In contrary,if dynamic system is unstable $\lambda_{1}+n\cdot h_{max}>0$, the error estimation boundary will be very huge. However, this phenomenon does not indicate that error estimation is not reliable, in contrast, it works well. We will use Example \ref{ex:1} to explain this situation.

As we know, in Example \ref{ex:1}, the exact solution is $x^{*}=e^{t}$, if there is an error $\varepsilon$ in numerical solution, satisfying $\widetilde{x}(t)=(1+\varepsilon)e^{t}-\varepsilon$. The residual will easily be calculated as $\delta(t)=\varepsilon$. Moreover, the forward error is $\Delta x=x^{*}-\widetilde{x}=\varepsilon(1-e^{t})$. It's obviously that the forward error of this example will be inevitable exponential growth, no matter what solver you choose. By Theorem \ref{thm:main}, as $\lambda_{1}=1$ and $h_{max}=0$ , our error estimation $\varepsilon(e^{t}-1)$ is also exponential increasing and exactly same as the forward error boundary $\|\Delta x\|_{\infty}$. It means that in general it is a sharp bound and  it is difficult to improve this bound.

\section{Examples}
In order to show performance of our error estimation to IVPS of Linear ODEs, in this section, we firstly use Theorem \ref{thm:main} to deal with a time-invariant system example. Then we give a time-variant stable system example. At last, we provide a time-variant unstable example from automobile suspension system with an adjustable damping model.

\subsection{Time-invariant System Example}
Consider this initial-value problem:
\[
   {\bm x}(t)'={\bm A}{\bm x}(t)+{\bm q}(t),{\bm A}=\left( \begin {array}{cc}
                    -1& 0\\
                     0 & 2
                   \end {array} \right), {\bm x}(0) = 0.
\]

In order to get the exact solution easily, let ${\bm x}= ( cos(\pi t)-1, sin(\pi t) )$. Then we obtain $\bf{ q}=\bf{ x}'-\bf{ A}\cdot \bf{ x}$.

Draw a figure by MATLAB as follow, in which the forward error $\Delta x_{1}(t)$ (the blue curve) almost is stable to fluctuate around zero, the forward error $\Delta x_{2}(t)$ (the red curve) is inevitable exponential growth. Our method of error estimation is exponential growth too, and it does gives an upper bound of the forward error all the time. Note, in our examples, we choose $\delta_{max}$ at each step and $h_{max}$ over the period.

\begin{figure}[H]

  \includegraphics[height=10cm, width=12cm]{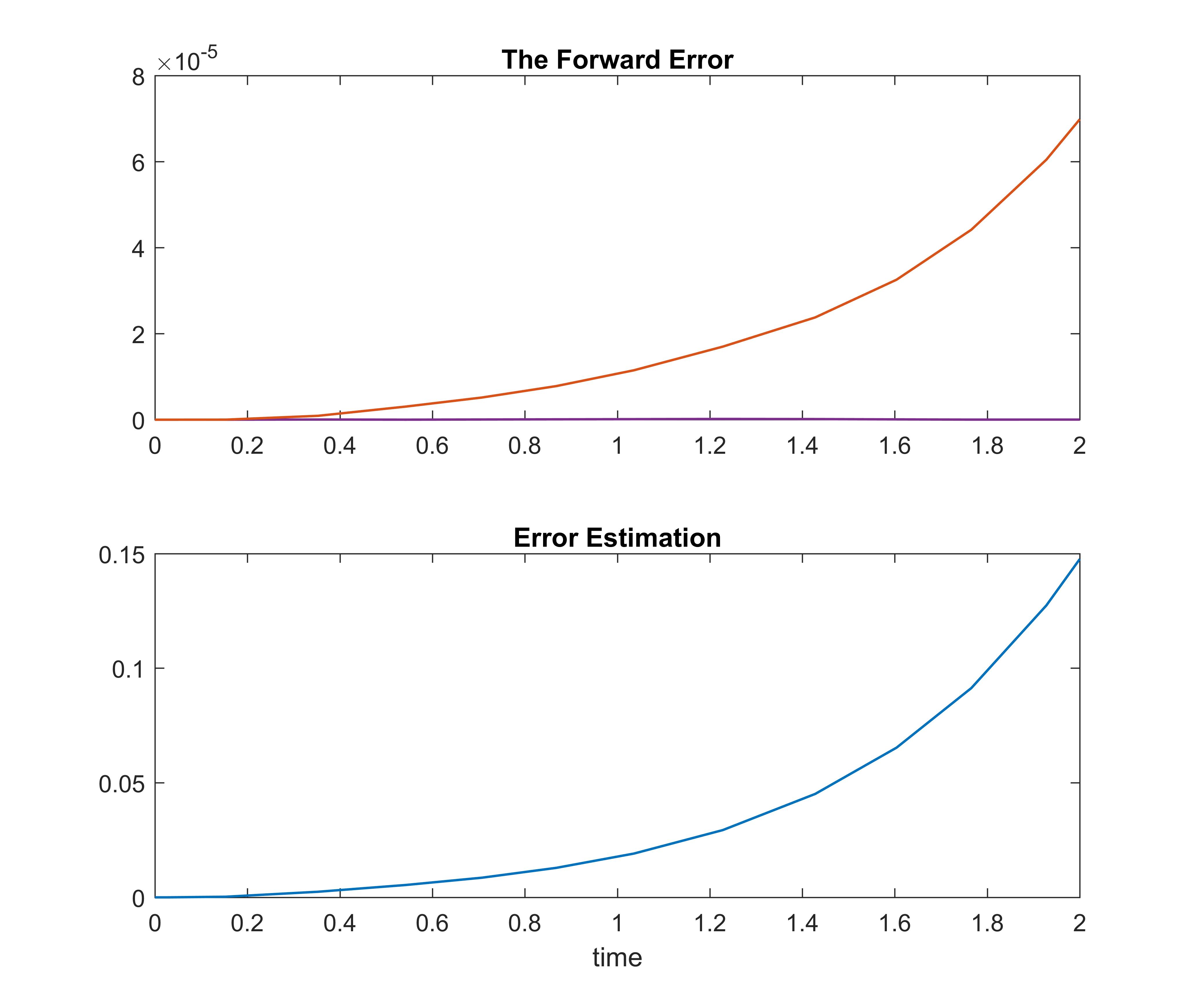}\\
  \caption{Time-invariant System Error Estimation}\label{example2}
\end{figure}

\subsection{A Time-variant Stable System Example}
Consider there is a linear dynamic system, its relationship as follow:
 \[\bf{ A}(t)=\left( \begin {array}{cc}
                    -6+0.2\cdot sin(2\pi t) & 0\\
                     1 & -5-0.1\cdot sin(\pi t)
                   \end {array} \right)  \ \mbox{ and }
  {\bm x}=\left( \begin{array}{c}cos(\pi t)-1\\sin(\pi t) \end {array} \right)\]
  In order to get the exact solution, we suppose the relationship $\bf{ q}=\bf{ x}'-\bf{ A}(t)\cdot \bf{ x}$. At the period $t\in[0,6]$, it's clear the initial ${\bm x}(0)=\left(\begin{array}{c}0\\0 \end{array} \right)$.

  We first use ode$45$ to solute the ODEs, we can get the numerical solution. Secondly, to deal with the numerical solution by hermit cubic spline, it's easy to calculate the residual of  numerical solution. Finally, according to theorem 3.3, we can do error estimation of numerical solution and draw a figure as Fig.({\ref{fig:stable}}).
  \begin{figure}[H]\label{fig:stable}
  \includegraphics[height=10cm, width=12cm]{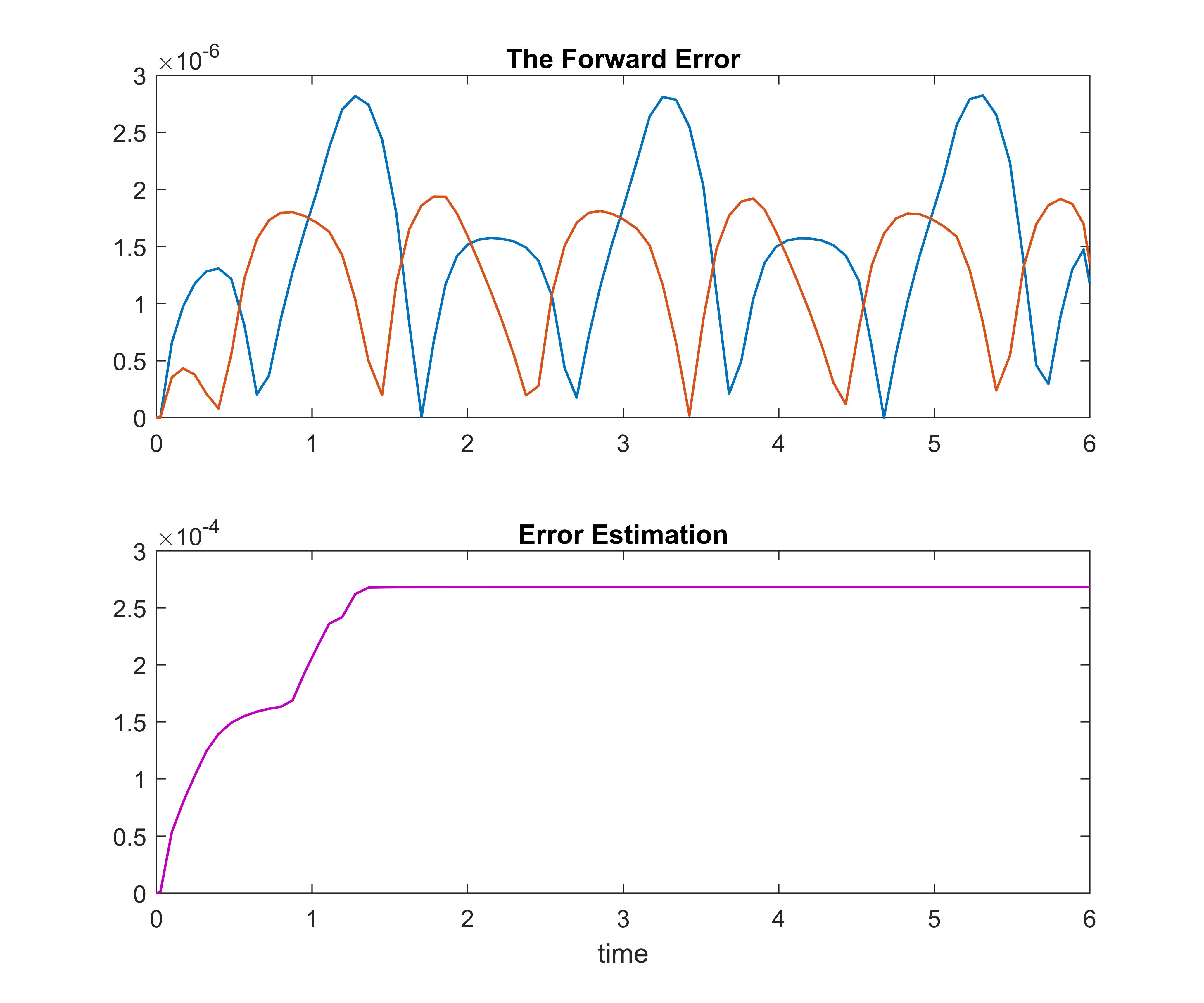}\\
  \caption{A Time-variant Stable System Error Estimation}
\end{figure}
     In this figure, we can see the boundary of error estimation is almost bigger than the boundary of forward error all the time. And both of them are stable as to the system is stable.
\subsection{Automobile Suspension System with an Adjustable Damping Absorber
}
While a car go though a ramp, we want it smooth enough, usually we use computer aided engineering software to simulate and to predict movement trends. However, the numerical solution of this dynamic system has error, which feedback to control-system will bring catastrophic result. It's necessary to calculate the forward error of numerical solution in order to warn control-system.The following we give an example of error estimation to automobile suspension system.

 Suppose: in the vertical direction, part mass of car $m1=500kg$, shock-absorbing spring $k1=1500N/m$, the mass of tire $m2=100kg$, the elasticity coefficient of tire $k2=2000N/m$.Additionally,  almost automobiles on road have active suspension, that's to say the damping will be changed by control system. Suppose the damping  is changed as $c1=200+200\cdot cos(2\pi t) N\cdot m/s^{2}$ in our example.
     The ODEs of this system can be write as $\bf{ x}'=\bf{ A}(t)\cdot \bf{ x}+\bf{ q}$. Where $ \bf{ x}=\left( \begin{array}{c}x_{1}\\x_{1}'\\x_{2}\\x_{2}'\end{array} \right)$ are part of car's displacement and velocity, tire's displacement and velocity, respectively.
     While go though a ramp, during $t\in[0,0.5]$, our goal is  $\bf{ x}=\left( \begin{array}{c}0.1\cdot sin(\pi t)\\0.1\cdot\pi cos(\pi t)\\0.01\cdot sin(\pi t)\\0.01\cdot\pi cos(\pi t)\end{array} \right)$, and initial value is  $\bf{ x}(0)=\left( \begin{array}{c}0\\0.1\cdot\pi \\0\\0.01\cdot\pi\end{array} \right)$, moreover, we can get $\bf{ q}=\bf{ x}'-\bf{ A}(t)\cdot \bf{ x}$.
\begin{figure}[H]
 \begin{minipage}[t]{0.5\linewidth}
  \centering
  \includegraphics[height=6cm,width=6cm]{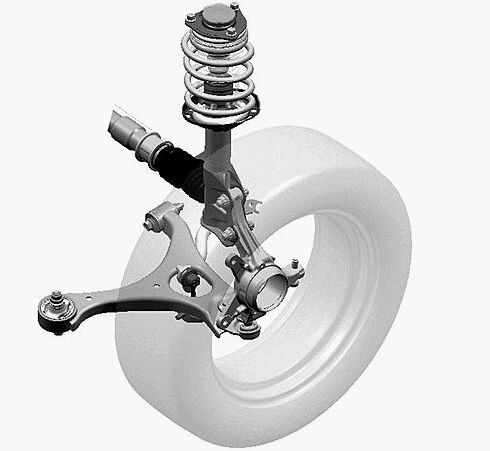}

  \end{minipage}
  \begin{minipage}[t]{0.5\linewidth}
  \centering
  \includegraphics[height=6cm,width=6cm]{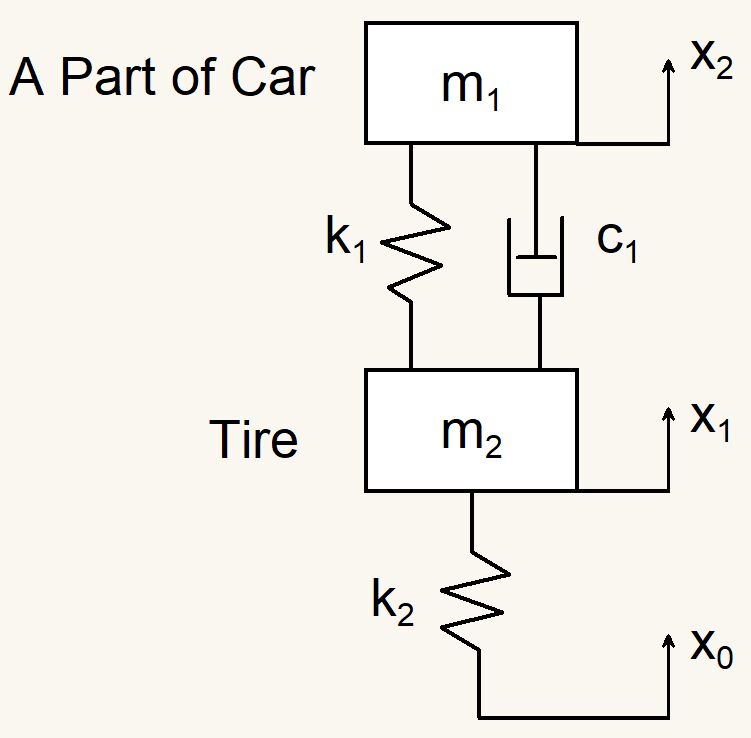}

  \end{minipage}
   \caption{Automobile Suspension System}

\end{figure}

 By using Taylor expansion $cos(x)=1-\frac{1}{2!}x^{2}+o(x^{2})$, we can get
 \[\bf{ A}(t)=\left( \begin {array}{cccc}
                    0 & 1 &0&0\\
                     3& \frac{4}{5}-\frac{\pi^{2}\cdot t^{2}}{250}&-3&\frac{\pi^{2}\cdot t^{2}}{250}-\frac{4}{5} \\
                     0&0&0&1\\
                     \frac{\pi^{2}\cdot t^{2}}{50}-4&15&35&4-\frac{\pi^{2}\cdot t^{2}}{50}\\
                   \end {array} \right) \]

As to the automobile suspension system is a time-variant unstable system, the forward error of numerical solution exponential growth by time, it's the same as the error estimation. That's to say, to unstable system ,if we need error estimation accuracy enough, the time should be short enough.
 \begin{figure}[H]
  \includegraphics[height=10cm, width=12cm]{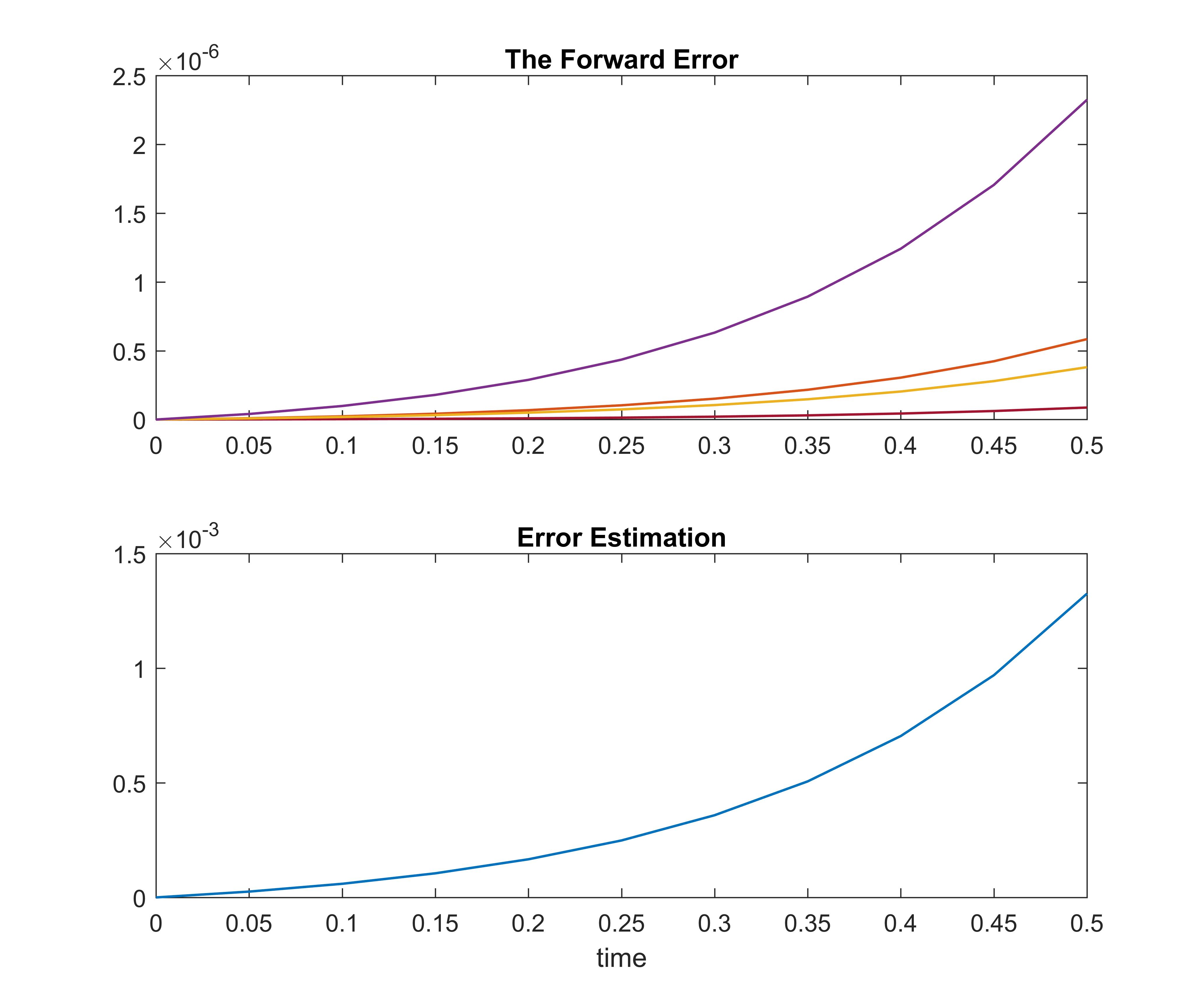}\\
  \caption{Automobile Suspension System}
\end{figure}

\section{Conclusion}
In this paper, we give a definition of residual, which accurately passes each point of a numerical solution. It has a tight  relationship with the forward error. Although, as indicated in Theorem \ref{thm:cubicHermiterror}, the Hermite cubic spline has error, the forward error satisfies the ODE (\ref{eq:forward error solution}) by Lemma \ref{lem:forward}.

For error estimation, we give a bound of time-invariant or time-variant system without exact solution. The bound of the estimation mainly depends on biggest eigenvalue, often smaller than Lipschitz constant. Verified by examples, our estimation method has better performance in unstable system. Moreover, if the backward error of numerical solution is smaller, that's to say the numerical solver is better or it uses a smaller step size, the error estimation will be more accurate accordingly.

\section{Acknowledgement}
This work was partially supported by NSFC 11471307 and CAS Research Program of Frontier Sciences
(QYZDB-SSW-SYS026).

\end{document}